\documentclass[11pt]{amsart}
\usepackage{amssymb,amscd,amsthm,amsmath,amsfonts,eufrak}

\newcommand{\norm}[1]{\mbox{$\|#1\|$}}
\newcommand{\x}{\times}
\newcommand{\cs}{\mbox{$C^{*}$-algebra}}
\newcommand{\css}{\mbox{$C^{*}$-algebras}}

\newcommand{\C}{\textbf{C}}%{\mathbb{C}}

\newcommand{\R}{\textbf{R}}%{\mathbb{R}}
\newcommand{\T}{\mathcal{T}}
\newcommand{\ov}[1]{\mbox{$\overline{#1}$}}

\newcommand{\al}{\mbox{$\alpha$}}
\newcommand{\eps}{\mbox{$\epsilon$}}

\newcommand{\ga}{\mbox{$\gamma$}}

\newcommand{\la}{\mbox{$\lambda$}}

\newcommand{\bc}{\begin{center}}

\newcommand{\ec}{\end{center}}
\newcommand{\be}{\begin{enumerate}}
\newcommand{\ee}{\end{enumerate}}
\newcommand{\beq}{\begin{equation}}
\newcommand{\eeq}{\end{equation}}
\newcommand{\beqn}{\begin{eqnarray}}
\newcommand{\eeqn}{\end{eqnarray}}
\newcommand{\beqns}{\begin{eqnarray*}}
\newcommand{\eeqns}{\end{eqnarray*}}

\newcommand{\bi}{\begin{itemize}}
\newcommand{\ei}{\end{itemize}}
\newcommand{\bd}{\begin{description}}
\newcommand{\ed}{\end{description}}

\newcommand{\lan}{\mbox{$\langle$}}
\newcommand{\ran}{\mbox{$\rangle$}}

\newtheorem{theorem}{Theorem}[section]

\newtheorem{definition}[theorem]{Definition}
\newtheorem{proposition}{Proposition}
\newtheorem{cor}[theorem]{Corollary}

\theoremstyle{remark}

\begin{document}
\title{The stabilization theorem for proper groupoids}
\author{Alan L. T. Paterson}
\address{University of Colorado, Department of Mathematics, Boulder, Colorado 80309-0395}
\email{apat1erson@gmail.com}
%\subjclass{Primary: 19K35, 22A22, 46L80, 58B34}
\date{August, 2009}
\begin{abstract}
The stabilization theorem for $A$-Hilbert modules was established by G. G. Kasparov.  The equivariant version, in which a locally compact group $H$ acts properly on a locally compact space $Y$, was proved by N. C. Phillips.  This equivariant theorem involves the Hilbert $(H,C_{0}(Y))$-module $C_{0}(Y,L^{2}(H)^{\infty})$.  It can naturally be interpreted in terms of a stabilization theorem for proper groupoids, and the paper establishes this theorem within the general proper groupoid context.  The theorem has applications in equivariant KK-theory and groupoid index theory.
\end{abstract}

\maketitle

\section{Introduction}

The Kasparov stabilization theorem (\cite{Kasp0}) asserts that for a C*-algebra $A$, the standard Hilbert module $A^{\infty}$ ``absorbs'' every other (countably generated) Hilbert $A$-module $P$ in the sense that
\[                  P\oplus A^{\infty}\cong A^{\infty}.                  \]
The theorem is of central importance for the development of KK-theory, and can be regarded as an extension of Swan's theorem for vector bundles.  Accounts of the theorem are given in the books by Blackadar and Wegge-Olsen (\cite{Blackadar,Wegge-Olsen}).  In \cite[Part 1, \S 2, Theorem 1]{Kaspconsp}, Kasparov obtained a stabilization theorem involving a group action: if $H$ is a locally compact group acting on $A$ and $P$ is a Hilbert $(H-A)$-module that is countably generated as a Hilbert $A$-module, then 
\[   P\oplus L^{2}(H,A)^{\infty}\cong L^{2}(H,A)^{\infty}                              \]
in the sense that there exists an $H$-continuous isomorphism from $P\oplus L^{2}(H,A)^{\infty}$ onto $L^{2}(H,A)^{\infty}$.  The isomorphism, however, need not be equivariant.  An elegant, self-contained account of all of this is contained in the paper \cite{MingoPhillips} of J. A. Mingo and 
W. J. Phillips.

For an equivariant stabilization theorem, one needs a properness condition, and N. C. Phillips has obtained such a theorem in the case of group actions 
(\cite[Theorem 2.9]{Phillips}).  Here, a locally compact group $H$ is assumed to act properly on a locally compact Hausdorff space $Y$. This action gives in the obvious way an action of $H$ on the $\cs$ $C_{0}(Y)$.   A Hilbert $(H,C_{0}(Y)$-module is defined to be a Hilbert $C_{0}(Y)$-module with a compatible action of $H$ which is strong operator continuous - for the precise definition, see, for example, \cite[Definition 1]{Kasp0} or 
\cite[Definition 2.1]{MingoPhillips}.   The theorem then asserts that for any Hilbert $(H,C_{0}(Y))$-module $P$, there is an equivariant isomorphism of Hilbert $(H,C_{0}(Y))$-modules:
\[   P\oplus (C_{0}(Y)\otimes L^{2}(H)^{\infty})\cong C_{0}(Y)\otimes L^{2}(H)^{\infty}.                  \]
Phillips uses this stabilization theorem in his proof of the generalized Green-Rosenberg theorem (that equivariant K-theory (in terms of $H$-Hilbert bundles over $Y$) is the same as the K-theory of the transformation groupoid $\cs$).  The starting point for the present paper is the observation (below) that Phillips's stabilization theorem (and the generalized Green-Rosenberg theorem) can be expressed very naturally in terms of locally compact proper groupoids.   (Accounts of the theory of locally compact groupoids are given in \cite{rg,Patersonbook}.)  
Groupoid versions of these theorems are, of course, required for the development of groupoid equivariant KK-theory, as well as for 
index theory in noncommutative geometry (\cite{Connesbook}), in particular, to orbifold theory.  (In connection with the latter, the properness condition is automatically satisfied since the structure of an orbifold with underlying space $X$ is completely described by the Morita equivalence class of a {\em proper}, effective, \'{e}tale Lie groupoid with orbit space homeomorphic to $X$ (\cite[pp.19-23]{orbifoldstringy}).)  The groupoid stabilization theorem is also necessary for extending Higson's K-theory proof of the index theorem (\cite{Higson}) to the equivariant case.  

In this paper, we will prove the stabilization theorem for proper groupoids; the generalized Green-Rosenberg theorem will be discussed elsewhere.  The proof of this stabilization theorem follows similar lines to that of Phillips's stabilization theorem, but also requires groupoid versions of results of \cite{MingoPhillips}.   The main additional technical issues to be dealt with arise from the fact that, unlike the Hilbert bundles of \cite{Phillips}, the Hilbert bundles involved in this paper are not usually locally trivial.  Indeed, the $G$-Hilbert module $P_{G}$ for a proper groupoid $G$, whose Hilbert module $P_{G}^{\infty}$ of infinite sequences stabilizes (as we will see) all the other $G$-Hilbert modules, is associated with a $G$-Hilbert bundle that is not usually locally trivial.  

We now translate the Phillips stabilization theorem into groupoid terms.  We are given a locally compact group $H$ acting properly on the left on $Y$.   One forms the transformation groupoid $G=H\x Y$: so multiplication is given by composition - $(h',hy)(h,y)=(h'h,y)$ - and inversion by 
$(h,y)^{-1}=(h^{-1},hy)$.  The unit space of $H\x Y$ can be identified with $Y$, and the properness condition translates into the requirement that the groupoid be {\em proper}: the map $g\to (r(g),s(g))$ (i.e. $(h,y)\to (hy,y)$) is proper (inverse image of compact is compact).  The next objective is to interpret in groupoid terms the 
$C_{0}(Y)\otimes L^{2}(H)^{\infty}$ occurring in the Phillips stabilization theorem.  A dense pre-Hilbert $(G,C_{0}(Y))$-module of 
$C_{0}(Y)\otimes L^{2}(H)=C_{0}(Y,L^{2}(H))$ is $C_{c}(H\x Y)=C_{c}(G)$ - so for a general proper groupoid $G$, we should replace $C_{0}(Y)\otimes L^{2}(H)$ by the completion $P_{G}$ of the pre-Hilbert module $C_{c}(G)$. The stabilization theorem for proper groupoids is then:
\[                   P\oplus P_{G}^{\infty}\cong P_{G}^{\infty}                  \]
where $P$ is (in the appropriate sense) a $G$-Hilbert module.

All groupoids in the paper are assumed to be locally compact, Hausdorff, proper and second countable, and all Hilbert spaces and Hilbert modules second countable.

For lack of a convenient reference, we state the following elementary partition of unity result which is proved as in, for example, \cite[Theorem 1.3]{Helgason}.  {\em Let $X$ be a second countable locally compact Hausdorff space, $C$ a compact subset of $X$ and $\{V_{1}, \ldots ,V_{n}\}$ a cover of $C$ by relatively compact, open subsets of $X$.  Then there exist $f_{i}\in C_{c}(V_{i})\subset C_{c}(X)$ with $0\leq f_{i}\leq 1$, $\sum_{i=1}^{n}f_{i}(y)\leq 1$ for all $y\in Y$, $\sum_{i=1}^{n}f_{i}(y)=1$ for all $y\in C$.}

\section{Groupoid Hilbert bundles}

We start by discussing the class of Hilbert bundles that we will need for $G$-actions.  The correspondence between Hilbert bundles over $Y$ and Hilbert $C_{0}(Y)$-modules seems to be well known, but for lack of a reference we sketch the details that we will need.  (Note that a Hilbert $C_{0}(Y)$-module $P$ can be regarded as a left $C_{0}(Y)$-module - $fp$ is the same as $pf$ for $p\in P, f\in C_{0}(Y)$.)  In the transformation groupoid case developed by Phillips, one uses locally trivial bundles with fiber $L$ and structure group $U(L)$ with the strong operator topology.  However, as noted above, the bundle associated with $C_{c}(G)$, required for the groupoid stabilization theorem, is not always locally trivial (though in the transformation groupoid case, it is trivial ($=Y\x L^{2}(H)$)), and we extend the class of bundles to be considered as follows.    Our approach, based on the work of Fell and Hoffman, is modelled on the account of the Dauns-Hoffman theorem in \cite{DupreG} with bundles of Banach spaces and $\css$ replaced by Hilbert bundles over $Y$ and Hilbert $C_{0}(Y)$-modules.  For the results of \cite[Chapter 2]{DupreG}, the Banach modules are modules over $C_{b}(X)$ where $X$ is completely regular.  In our case, we wish to obtain similar results for Hilbert modules over $C_{0}(Y)$.   (The corresponding modifications needed for $C_{0}(Y)$-algebras are given in 
\cite{Patdesc}.  See also \cite[C.2]{Williams}.)  Since the Hilbert bundles that we will need are usually not locally trivial, it is natural to define such a bundle in terms of a space of sections deemed to be continuous and vanishing at infinity (cf. \cite[Ch. 10]{D2}).   This can be done.  However, for our purposes, it is more convenient to use a topological approach which is in some respects akin to the classical definition of vector bundles.  In the following definition of {\em Hilbert bundle}, we are given a topology on the total space and the set of continuous sections that vanish at infinity has to satisfy certain properties. 

\begin{definition}  \label{def:Hilbertbundle} 
Let $\{H_{y}\}_{y\in Y}$ be a family of Hilbert spaces, $E$ a second countable, topological space which is the disjoint union of the $H_{y}$'s, and $\pi:E\to Y$ be the projection map.  Let $C_{0}(Y,E)$ be the set of continuous sections $F$ of $E$ such $\lim_{y\to \infty} \norm{F(y)}=0$.   Then $E$ is called a {\em Hilbert bundle over $Y$} if the following properties hold: 
\be
\item [(i)] the addition map $E\oplus_{Y} E\to E$ and the scalar multiplication map $(Y\x \C)\oplus_{Y} E\to E$ are continuous;
\item [(ii)] For each $F\in C_{0}(Y,E)$, the map $y\to \norm{F(y)}$ is continuous; 
\item [(iii)] for each $y$, $\{F(y):F\in C_{0}(Y,E)\}=H_{y}$. 
\item [(iv)] The topology on $E$ is determined by $C_{0}(Y,E)$ in the sense that a base for it is given by the sets of the form $U_{F,\eps}$, where $U$ is an open subset of $Y$ and 
\begin{equation}   \label{eq:UFeps}
U_{F,\eps}=\{h_{y}: y\in U, h_{y}\in H_{y}, \norm{h_{y} - F(y)}<\eps\}.     
\end{equation}
\ee
\end{definition} 

Here are some comments on the preceding definition.  From (i) and (iii), $C_{0}(Y,E)$ is a vector space.  It follows from (iv) and (iii) that 
$\pi$ is open and continuous, and each $H_{y}$ has its Hilbert space norm topology in the relative topology of $E$.  Using (ii), (iii) and (iv), the norm function $\norm{.}:E\to \R$ is continuous.  By a simple triangular inequality argument - use the continuity of $y\to \norm{F(y)-F'(y)}$ for $F, F'\in C_{0}(Y,E)$ - if $\xi\in H_{y_{0}}$ and $F\in C_{0}(Y,E)$ is fixed such that $F(y_{0})=\xi$, then the family of sets $U(F,\eps)$ with 
$y_{0}\in U$, $\eps>0$, is a base of neighborhoods for $\xi$ in $E$.  By \cite[p.57]{Kelley}, there is a countable base for the topology of $E$ consisting of sets of the form $U(F,\eps)$.  We note that $E$ is Hausdorff though we will not use this fact.  We also note that in (iv), we get the same topology if the functions $F$ are restricted to lie in a subspace of $C_{0}(Y,E)$ which is dense in the uniform norm topology (below).

\begin{proposition}   \label{prop:secondctble}
Let $E$ be a Hilbert bundle over $Y$.  Then $C_{0}(Y,E)$ is a separable $C_{0}(Y)$-Hilbert module in the uniform norm topology: 
$\norm{F}=\sup_{y\in Y}\norm{F(y)}$.
\end{proposition}
\begin{proof}
To show that $C_{0}(Y,E)$ is a Banach space, one modifies the proof for the corresponding elementary result on uniform convergence of functions.  Let $\{F_{n}\}$ be a Cauchy sequence in $C_{0}(Y,E)$.  Then $F_{n}\to F$ pointwise for some section $F$ of $E$.  We now show that $F\in C_{0}(Y,E)$.  It is obvious that $\norm{F(y)}\to 0$ as $y\to \infty$.  It remains to show that $F$ is continuous.  Let $y_{k}\to y_{0}$ in $Y$.  We have to show that 
$F(y_{k})\to F(y_{0})$.  Let $F'\in C_{0}(Y,E)$ be such that $F'(y_{0})=F(y_{0})$.  Let $U$ be an open neighborhood of $y_{0}$ and $\eps>0$.  
One shows that eventually, $F(y_{k})\in U(F',\eps)$ and the continuity of $F$ follows by the preceding comments on the definition.  For $F_{1}, F_{2}\in C_{0}(Y,E)$, define $\lan F_{1},F_{2}\ran:Y\to \C$ in the obvious way: $\lan F_{1},F_{2}\ran (y)=\lan F_{1}(y),F_{2}(y)\ran$.  By the polarization identity and (ii) of the definition, $\lan F_{1},F_{2}\ran\in C_{0}(Y)$.  It is easy to check that $C_{0}(Y,E)$ is a Hilbert $C_{0}(Y)$-module with inner product $\lan .,.\ran$ and module action given by: $Ff(y)=f(y)F(y)$.  

We now prove that $C_{0}(Y,E)$ is separable.  Let $\mathcal{A}$ be a countable base for $E$ whose elements are of the form $U(F,\eta)$.  It suffices to show that for a compact subset $C$ of $Y$, the space of sections $A\subset C_{0}(Y,E)$ with support in $C$ is separable.   Let $F'\in A$ and $\eps>0$.
For each $y\in C$, let $U_{y}$ be a relatively compact, open neighborhood of $y$ in $Y$.  Then $F'(y)\in U_{y}(F',\eps)$, and there exists a 
$V_{y}(F_{y},\eps_{y})\in \mathcal{A}$ such that $F'(y)\in V_{y}(F_{y},\eps_{y})\subset U_{y}(F',\eps)$.  In particular, $y\in V_{y}\subset U_{y}$ and
$\norm{F'(y') - F_{y}(y')}<\eps$ for all $y'\in V_{y}$.   Since $C$ is compact, there exists a finite cover $\{V_{y_{1}}, \ldots V_{y_{n}}\}$ of $C$.  Let $\{f_{i}\}$ ($1\leq i\leq n$) be a partition of unity for $C$ subordinate to the $\{V_{y_{i}}\}$, and let $F''=\sum_{i=1}^{n} f_{i}F_{y_{i}}$.  Then $\norm{F'(y) - F''(y)}<\eps$ for all $y\in Y$.  The span of such functions $F''$ in $C_{0}(Y,E)$ is separable, and the separability of $C_{0}(Y,E)$ then follows.
\end{proof}

As a simple example of a Hilbert bundle, let $Y=(0,2)$, $F$ be the trivial Hilbert bundle $Y\x \C^{2}$ and $\{e_{1}, e_{2}\}$
the standard orthonormal basis for $\C^{2}$.  Then $C_{0}(Y,F)=C_{0}((0,2))\x C_{0}((0,2))$ in the obvious way.  Let 
$E$ be the subbundle $[(0,1]\x \C e_{1}]\cup [(1,2)\x \C^{2}]$ of $F$ with the relative topology.  Then $E$ is a Hilbert subbundle of $F$ though it is neither locally constant nor locally compact.  (Note that $C_{0}(Y,E)$ can be identified with 
$C_{0}((0,2))\x \{f\in C_{0}((0,2)): f(y)=0 \mbox{ for $0<y\leq 1$}\}$.)

A morphism between two Hilbert bundles $E, F$ over $Y$ is (cf. \cite[Definition 1.5]{Phillips}) a continuous bundle map $\Phi:E\to F$ whose restriction $\Phi_{y}:E_{y}\to F_{y}$ for each $y\in Y$ is a bounded linear map and $\sup_{y\in Y}\norm{\Phi_{y}}=\norm{\Phi}<\infty$, and such that the adjoint map
$\Phi^{*}:F\to E$, where $\Phi^{*}(\xi_{y})=(\Phi_{y})^{*}(\xi_{y})$ for $\xi_{y}\in F_{y}$ is also continuous.  It is obvious that any such morphism $\Phi$ determines an adjointable Hilbert module map $\tilde{\Phi}:C_{0}(Y,E)\to C_{0}(Y,F)$ by setting $\tilde{\Phi}(F)(y)=\Phi_{y}(F(y))$.  It is also obvious that with these morphisms, the class of Hilbert bundles over $Y$ is a category.

We have seen that every $C_{0}(Y,E)$ is a second countable $C_{0}(Y)$-Hilbert module.  We will show that every second countable $C_{0}(Y)$-Hilbert module $P$ is of this form.  We recall first that a morphism between two Hilbert $C_{0}(Y)$-modules $P, Q$ is an adjointable map $T:P\to Q$.  This gives the category of Hilbert $C_{0}(Y)$-modules.   Two Hilbert $C_{0}(Y)$-modules $P, Q$ are said to be {\em equivalent} - written $P\cong Q$ - if there exists a unitary morphism $U:P\to Q$.  Next, a result of 
Kasparov (\cite[Theorem 1]{Kasp0}, \cite[Lemma 15.2.9]{Wegge-Olsen}) gives that in any Hilbert $A$-module $P$ and for any $p\in P$,
\begin{equation}   \label{eq:ppf} 
p=\lim_{\eps\to 0^{+}}p\lan p,p\ran [\lan p,p\ran + \eps]^{-1}.
\end{equation}
It follows by Cohen's factorization theorem and (\ref{eq:ppf}) that $P=\{fp: f\in C_{0}(Y), p\in P\}$.  
In the stabilization theorem of Kasparov, the Hilbert $A$-modules are assumed to be countably generated. It is obvious that in our situation
($P$ second countable) $P$ is automatically countably generated.  

Let $P$ be a $C_{0}(Y)$-Hilbert module.    We construct an associated Hilbert bundle $E$ in the familiar way (e.g. \cite{DupreG}).  For $y\in Y$, let $I_{y}=\{f\in C_{0}(Y): f(y)=0\}$, a closed ideal in $C_{0}(Y)$.  By Cohen's factorization theorem, $I_{y}P$ is closed in $P$.  Let $P/(I_{y}P)=P_{y}$.  We claim that the norm on $P_{y}$ is a Hilbert space norm, with inner product given by $\lan p+I_{y}P,q+I_{y}P\ran=\lan p,q\ran (y)$.  This inner product is well-defined.  To see that it is non-degenerate, suppose that $\lan p,p\ran(y)=0$.  Then 
$\lan p,p\ran\in I_{y}$ and
by (\ref{eq:ppf}), $p\in \ov{(I_{y}P)}=I_{y}P$, and non-degeneracy follows.  Let $E=\cup_{y\in Y}P_{y}$.  If we wish to emphasize the connection of $E$ with $P$, we write $E_{P}$ in place of $E$.  (If $Q$ is just a pre-Hilbert $C_{0}(Y)$-submodule, we define $E_{Q}$ to be $E_{\ov{Q}}$.) 
For each $p\in P$, let $\hat{p}(y)=p + I_{p}P\in H_{y}$.  We sometimes write $p_{y}$ in place of $\hat{p}(y)$.
For each open subset $U$ of $Y$ and each $\eps>0$, define $U_{p,\eps}=U_{\hat{p},\eps}$, the latter being defined as in (\ref{eq:UFeps}).

We now show that the functor $E\to C_{0}(Y,E)$ is an equivalence for the categories of Hilbert bundles over $Y$ and of 
Hilbert $C_{0}(Y)$-modules.  

\begin{proposition}   \label{prop:Etop}
Let $P$ be a Hilbert $C_{0}(Y)$-module.  Then the family of $U_{p,\eps}$'s ($p\in P$) is a base for a second countable topology $\T_{P}$ on $E$ which makes $E$ into a Hilbert bundle over $Y$.  Further, the map $p\to \hat{p}$ is a Hilbert $C_{0}(Y)$-module unitary from $P$ onto $C_{0}(Y,E)$, and the map $P\to E$ is an equivalence between the category of Hilbert $C_{0}(Y)$-modules $P$ and the category of Hilbert bundles $E$ over $Y$.
\end{proposition}
\begin{proof}
Give each $\hat{p}$ the uniform norm as a section of $E$.  The proposition is an easier version of corresponding results for Banach $A$-modules in \cite{DupreG}.  It is easier because, as earlier, by the polarization identity, the maps $y\to \norm{\hat{p}(y)}=\sqrt{\lan p,p\ran(y)}$ are continuous 
(instead of just upper semicontinuous) and vanish at infinity.   Then $\norm{\hat{p}}^{2}=\norm{\lan p,p\ran}=\norm{p}^{2}$, giving $p\to \hat{p}$ an isometry. 
We now check the conditions of Definition~\ref{def:Hilbertbundle} to show that $E$ is a Hilbert bundle over $Y$.  One easily checks that the family of $U_{p,\eps}$'s ($p\in P$) is a base for a topology $\T_{P}$ on $E$, each $\hat{p}$ is continuous and the addition and scalar multiplication maps for $E$ are continuous.  The topology $\T_{P}$ on $E$ is second countable since $P$ is.  This gives (i) of Definition~\ref{def:Hilbertbundle}, while (iii) of that definition is trivial.  The remaining requirements, (ii) and (iv) will follow once we have shown that $\hat{P}=C_{0}(Y,E)$.   As in the proof of Proposition~\ref{prop:secondctble} (cf. \cite[Proposition 2.3]{DupreG}) $\hat{P}$ is dense in $C_{0}(Y,E)$.
Further, $\lan \hat{p},\hat{q}\ran = \lan p,q\ran$ 
giving the map $p\to \hat{p}$ unitary.  Then $\hat{P}=C_{0}(Y,E)$ since the map $p\to \hat{p}$ is isometric and $P$ is complete.  A morphism $T:P\to Q$ of Hilbert $C_{0}(Y)$-modules determines a Hilbert bundle morphism $\Phi=\Phi_{T}:E_{P}\to E_{Q}$ in the natural way:  set $\Phi=\{T_{y}\}$ where $T_{y}$ is defined: $T_{y}p_{y}=(Tp)_{y}$.  Then $\Phi:E_{P}\to E_{Q}$ is a continuous bundle map, and $\norm{\Phi}=\norm{T}$.  
\end{proof}

For a Hilbert bundle $E$ over $Y$, let $G*E=\{(g,\xi): s(g)=\pi(\xi)\}$ with the relative topology inherited from $G\x E$.  Then $E$ is called a 
{\em $G$-Hilbert bundle} if there is a continuous map $(g,\xi)\to g\xi$ from $G *E\to E$ which is algebraically a left groupoid action (by unitaries).   (The unitary condition means that for each fixed $g\in G$, the map $\xi\to g\xi$ is unitary from $H_{s(g)}$ onto $H_{r(g)}$.)  One can also define this notion in terms of pull-back bundles as in \cite{LeGall1,LeGall2}, but the approach adopted here is more elementary, and closer in spirit to the usual definition of a group Hilbert bundle.   A Hilbert $C_{0}(Y)$-module $P$ is called a {\em $G$-Hilbert module} if $E_{P}$ is a $G$-Hilbert bundle. 
The corollary to the following proposition shows that when $G$ is a transformation groupoid $H\x Y$, a $G$-Hilbert module is the same as a Hilbert $(H,C_{0}(Y))$-module in the notation of \cite{Phillips}.  (In \cite[Proposition 1.3]{Phillips}, it is shown that if $E$ is an $H$-Hilbert bundle over $Y$, then $C_{0}(Y,E)$ is a Hilbert $(H,C_{0}(Y))$-module.  The corollary shows that the opposite direction holds as well as long as we use the wider category of Hilbert bundles of the present paper.)  

\begin{proposition}  \label{prop:cont}
A left groupoid action of $G$ on $E$ is continuous if and only if, for each $F\in C_{0}(Y,E)$, the map $g\to gF_{s(g)}$ is continuous from $G\to E$.
\end{proposition}
\begin{proof}
If the action is continuous, then trivially, the maps $g\to gF_{s(g)}$ are continuous.  The converse is very similar to \cite[Corollary 1]{Patdesc}, and so we give only a brief sketch of the proof.  Suppose then that for each $F\in C_{0}(Y,E)$, the map $g\to gF_{s(g)}$ is continuous from $G\to E$.  Let $\{g_{n}\}$ be a sequence in $G$ and $\{\xi_{n}\}$ a sequence in $E$ with $\xi_{n}\in E_{s(g_{n})}$ such that $g_{n}\to g$ in $G$ and $\xi_{n}\to \xi$ in $E$.  We have to show that $g_{n}\xi_{n}\to g\xi$ in $E$.  By Definition~\ref{def:Hilbertbundle},(iii), there exist $F\in C_{0}(Y,E)$ such that $g\xi=F_{r(g)}$ and $F'\in C_{0}(Y,E)$ such that $\xi=F'_{s(g)}$.  Then $\norm{\xi_{n} - F'_{s(g_{n})}}\to 0$, so that 
$\norm{g_{n}\xi_{n} - g_{n}F'_{s(g_{n})}}\to 0$ as well. Next, by assumption,  $g_{n}F'_{s(g_{n})}\to gF'_{s(g)}=g\xi=F_{r(g)}$ and so by the continuity of $F$, $\norm{g_{n}F'_{s(g_{n})} - F_{r(g_{n})}}\to 0$.  So $g_{n}\xi_{n}\to g\xi$.
\end{proof}
\begin{cor}   \label{cor:Hmodule}
Let $G$ be a transformation groupoid $H\x Y$.  Then the map $E\to C_{0}(Y,E)$ is an equivalence between the category of $H$-Hilbert bundles over $Y$ and the category of Hilbert $(H,C_{0}(Y))$-modules. 
\end{cor}
\begin{proof}
We recall (\cite{Kasp0,MingoPhillips}) that a Hilbert $C_{0}(Y)$-module $S$ is an $(H,C_{0}(Y))$-module if it is a left $H$-module such that 
$h(Ff)=(hF)(hf)$, the map $h\to hF$ is continuous, and $\lan hF,hF'\ran = h\lan F,F'\ran$ for all $h\in H, F,F'\in S$ and $f\in C_{0}(Y)$.  
(Of course, $(hf)(y)=f(h^{-1}y)$.)  An $H$-Hilbert bundle over $Y$ (cf. \cite[Definition 1.2]{Phillips}) 
is a Hilbert bundle over $Y$ (in the sense of this paper) with a continuous action $(h,\xi)\to h\xi$ from $H\x E$ into $E$ such that for each $y$, the action of $h$ on $E_{y}$ is a unitary onto $E_{hy}$.  (Recalling that $(H\x Y)_{y}=H$ for all $y$, it is obvious that $H$-Hilbert bundles over $Y$ are just the same as the groupoid $(H\x Y)$-Hilbert bundles.)  Suppose, first that $E$ is an $H$-Hilbert bundle.  Then (as in \cite[Proposition 1.3]{Phillips}) the Hilbert $C_{0}(Y)$-module $C_{0}(Y,E)$ is a Hilbert $(H,C_{0}(Y))$-module, where $(fF)(y)=f(y)F(y)$ and $(hF)(y)=h[F(h^{-1}y)]$ ($F\in C_{0}(Y,E)$).  For the converse, let $P$ be a Hilbert $(H,C_{0}(Y))$-module, $E=E_{P}$.  By Proposition~\ref{prop:Etop}, we can canonically identify $P$ with $C_{0}(Y,E)$.   It is obvious that $hI_{y}=I_{hy}$.  We define a groupoid action of $H\x Y$ on $E$ by setting $(h,y)(p+I_{y}P)=hp+I_{hy}$, i.e. $(h,y)p_{y}=(hp)_{hy}$.  We now check that this is indeed a groupoid action (in the sense of this paper).  The algebraic properties are obvious using the formulas for multiplication and inversion in $H\x Y$ given in the introduction.  To prove that $H\x Y$ acts on $E$ by unitaries, 
\[   \lan (h,y)p_{y},(h,y)q_{y}\ran = \lan hp,hq\ran(hy)=\lan p,q\ran (h^{-1}hy)=\lan p_{y},q_{y}\ran.  \]
Last, to prove the continuity of the groupoid action on $E$, we have, by Proposition~\ref{prop:cont}, to show, identifying $\hat{P}$ with 
$C_{0}(Y,E)$, that for each $p\in P$, the map $(h,y)\to (h,y)p_{y}$ is continuous from $H\x Y$ into $E$, i.e. that the map $(h,y)\to (hp)_{hy}$ is continuous.  This is simple to prove using the continuity of the map $h\to hp$ in $P$.
\end{proof}

If $P, Q$ are $G$-Hilbert modules, then a Hilbert $C_{0}(Y)$-module morphism $T:P\to Q$ is called {\em $G$-equivariant} if for all $g\in G$, $T_{r(g)}g=gT_{s(g)}$ on $(E_{P})_{s(g)}$.  Using the fact that the groupoid action is unitary, $T^{*}$ is also $G$-equivariant.  Of course, $P$ and $Q$ are said to be {\em equivalent} ($P\cong Q$) if there exists $G$-equivariant unitary between them.

A pre-Hilbert $C_{0}(Y)$-module $Q$ is called a {\em pre-$G$-Hilbert module} if $\ov{Q}$ is a $G$-Hilbert module, and the action of $G$ on $E=E_{\ov{Q}}$ leaves invariant the $Q_{y}$'s, where $Q_{y}$ is the image of $Q$ in $E_{y}$.  As we will see below, an important example of a pre-$G$-Hilbert module is the case $Q=C_{c}(G)$.  The $C_{0}(Y)$-module action on $C_{c}(G)$ is given by: $(F,f)\to F(f\circ r)$ and the $C_{0}(Y)$-valued inner product on $C_{c}(G)$ by: $\lan F_{1},F_{2}\ran (y)=\lan (F_{1})_{y}, (F_{2})_{y}\ran$ ($F_{y}=F_{\mid G^{y}}$).  One uses the axioms for a locally compact groupoid to check the required properties.  For example, the continuity of $y\to \lan (F_{1})_{y}, (F_{2})_{y}\ran$ follows from the axiom that for $\phi\in C_{c}(G)$, the function $y\to \int_{G^{y}}\phi(g)\,d\la^{y}(g)$ is continuous.  Let $P_{G}$ be the Hilbert $C_{0}(Y)$-module completion
of $C_{c}(G)$, and $L^{2}(G)=E_{P_{G}}$,  the Hilbert bundle determined by $P_{G}$ as in Proposition~\ref{prop:Etop}.  It is easy to check that for each $y$, the image of $C_{c}(G)$ in $H_{y}$ is naturally identified as a pre-Hilbert space with $C_{c}(G^{y})$ with the $L^{2}(G^{y})$ inner product.  So 
the Hilbert space $(E_{P_{G}})_{y}=L^{2}(G^{y})$ (which justifies writing $E_{P_{G}}$ as $L^{2}(G)$).   The isomorphism $F\to \hat{F}$ from $C_{c}(G)$ into 
$C_{c}(Y,L^{2}(G))$ takes $F$ to the section $y\to F_{y}=\hat{F}(y)$, and the family of sets $U(F,\eps)$ forms a base for the topology of $L^{2}(G)$.   The $G$-action on $L^{2}(G)$ is the natural one: $g\xi_{s(g)}(h)=\xi_{s(g)}(g^{-1}h)$ ($h\in G^{r(g)}$) for $\xi_{s(g)}\in L^{2}(G^{s(g)})$.  We now show that this action is continuous for the topology of $L^{2}(G)$.

\begin{proposition}  \label{prop:L2Gactioncont}
The $G$-action is continuous on $L^{2}(G)$ (so that $L^{2}(G)$ is a $G$-Hilbert bundle and $P_{G}$ a $G$-Hilbert module).
\end{proposition}
\begin{proof}
From Proposition~\ref{prop:cont}, it suffices to show that if $\psi\in C_{0}(Y,L^{2}(G))$ and $g_{n}\to g$ in $G$, then 
$g_{n}\psi_{s(g_{n})}\to g\psi_{s(g)}$.  Since $\widehat{C_{c}(G)}$ is uniformly dense in $C_{0}(Y,L^{2}(G)) = \widehat{P_{G}}$ 
(Proposition~\ref{prop:Etop}), we can suppose that $\psi=\hat{F}$ where $F\in C_{c}(G)$.
By Tietze's extension theorem, there exists $F'\in C_{c}(G)$ such that $F'_{r(g)}=gF_{s(g)}$.  It is sufficient, then, to show that
$\norm{F'_{r(g_{n})} - g_{n}F_{s(g_{n})}}_{2}\to 0$ since the $U(F',\eps)$'s ($r(g)\in U$) form a base of neighborhoods for $F'_{r(g)}$ in $L^{2}(G)$.  Arguing by contradiction, suppose that the sequence \\
$\{\norm{F'_{r(g_{n})} - g_{n}F_{s(g_{n})}}_{2}\}$ does not converge to $0$.  We can then suppose that for some $k>0$, $\norm{F'_{r(g_{n})} - g_{n}F_{s(g_{n})}}_{2}\geq k$ for all $n$.  Let $D$ be a compact subset of $G$ containing the sequence $\{g_{n}\}$ and let 
$C=D supp(F)\cup supp (F')\subset G$.  Since $C$ is compact, $M=\sup_{u\in Y}\la^{u}(C^{u})<\infty$.  Then 
\[  [\sup\{\mid F'(h) - F(g_{n}^{-1}h)\mid: h\in G^{r(g_{n})}\cap C\}]^{2}M \geq \norm{F'_{r(g_{n})} - g_{n}F_{s(g_{n})}}^{2}_{2}\geq k^{2}. \]
So we can find $h_{n}\in G^{r(g_{n})}\cap C$ such that $\mid F'(h_{n}) - F(g_{n}^{-1}h_{n})\mid > \sqrt{k^{2}/(2M)}$.   By the compactness of $C$, we can suppose that 
$h_{n}\to h\in G^{r(g)}$, and thus obtain $\mid F'(h) - gF_{s(g)}(h)\mid >0$, contradicting $F'_{r(g)}=gF_{s(g)}$.
\end{proof}

$C_{0}(Y)$ itself is naturally a $G$-Hilbert module.  To see this, $C_{0}(Y)$ is, like every $\cs$, a Hilbert module over itself.  The Hilbert bundle determined by $C_{0}(Y)$ is, of course, just $Y\x \C$.  It is left to the reader to check that the topology determined on $E=Y\x \C$ is just the product topology.  The $G$-action on $Y$ is given by $(g,s(g),a)\to (r(g),a)$ (trivially continuous).

Let $E(i)=\{E(i)_{y}\}$ ($1\leq i\leq n$) be Hilbert bundles over $Y$ and $P(i)$ the Hilbert $C_{0}(Y)$-module $C_{0}(Y,E(i))$.  Let 
$E=E_{\oplus_{i=1}^{n}P(i)}$.  It is easy to check that $E=\oplus_{i=1}^{n} E(i)$ with the relative topology inherited from $E(1)\x \ldots E(n)$. (Note also that the elements of $C_{0}(Y,E)$ are of the form $F=(F_{1}, \ldots ,F_{n})$ where $F_{i}\in C_{0}(Y,E(i))$.)
Similarly if $1\leq i<\infty$, then $E=\oplus_{i=1}^{\infty} E(i)$ is defined to be $E_{\oplus_{i=1}^{\infty}P(i)}$.  
(Here (e.g. \cite[2.2.1)]{LeGall1}) $\oplus_{i=1}^{\infty} P_{i}$ consists of all sequences $\{p_{i}\}$, $p_{i}\in P_{i}$, such that 
$\sum_{i=1}^{\infty}\lan p_{i},p_{i}\ran$ is convergent in $C_{0}(Y)$.  The argument of, for example, \cite[pp.237-238]{Wegge-Olsen}, shows that 
$\oplus_{i=1}^{\infty} P_{i}$ is a Hilbert $C_{0}(Y)$-module with $C_{0}(Y)$-valued inner product given by 
$\lan\{p_{i}\},\{q_{i}\}\ran = \sum_{i=1}^{\infty}\lan p_{i},q_{i}\ran$.)  Then for each $y$, $E_{y}$ is the Hilbert space direct sum $\oplus_{i=1}^{\infty} E(i)_{y}$.  Using Proposition~\ref{prop:Etop}, the topology on $E$ can be conveniently described in terms of convergent sequences: $\xi^{n}\to \xi$ ($\xi^{n}=\{\xi^{n}_{i}\}, \xi=\{\xi_{i}\}$) if and only if $\xi^{n}_{i}\to \xi_{i}$ in $E(i)$ for all $i$ and 
$\sum_{i=N}^{\infty} \norm{\xi^{n}_{i}}^{2}\to 0$ as $N, n\to \infty$.  When $E(i)=E(1)$ for all $i$, then we write $E=E(1)^{\infty}$, corresponding to the module $P=P(1)^{\infty}$.  Using the preceding criterion for convergent sequences, it is straightforward to show that the Hilbert bundles 
$\oplus_{i=1}^{n} E(i), \oplus_{i=1}^{\infty} E(i)$ are $G$-Hilbert bundles in the natural way if the $E(i)$'s are $G$-Hilbert bundles.  Of course, 
$\oplus_{i=1}^{n} P(i), \oplus_{i=1}^{\infty} P(i)$ are then $G$-Hilbert modules.

We also require that for any $G$-Hilbert module $P$,
\begin{equation}  \label{eq:pinftyinfty}
(P^{\infty})^{\infty}\cong P^{\infty}.
\end{equation}
To prove this, using the Cantor diagonal process, one ``rearranges'' a sequence 
$\{\xi_{i}\}\in (P^{\infty})^{\infty}$, $\xi_{i}=\{\xi_{ij}\}$, $\xi_{ij}\in P$, as a sequence in $P^{\infty}$, and checks that the 
$C_{0}(Y)$-Hilbert module structure and the $G$-action are preserved. 

A number of natural $G$-Hilbert $C_{0}(Y)$-modules arise from other such modules as tensor products over $C_{0}(Y)$ 
(cf. \cite{Blanchard92,Blanchard95}).    See \cite[3.2.2]{LeGall1} for a pull-back approach to the construction of tensor product $G$-Hilbert modules.  
Let $P, Q$ be pre-Hilbert $C_{0}(Y)$-modules and form the algebraic balanced tensor product $P\otimes_{alg,C_{0}(Y)} Q$.  This is a pre-Hilbert $C_{0}(Y)$-module in the natural way, i.e. with $(p\otimes q)f=p\otimes qf=p\otimes fq=pf\otimes q$ and inner product given by 
$\lan p_{1}\otimes q_{1}, p_{2}\otimes q_{2}\ran=\lan p_{1},p_{2}\ran\lan q_{1},q_{2}\ran$.
The completion of $P\otimes_{alg,C_{0}(Y)} Q$, quotiented out by the null space of the norm induced by the inner product, is a Hilbert $C_{0}(Y)$-module
$P\otimes_{C_{0}(Y)} Q$.  (When $P, Q$ are Hilbert modules, the construction is a special case of the inner tensor product $P\otimes_{\phi} Q$  (\cite[13.5]{Blackadar}) with $\phi: C_{0}(Y)\to B(Q)$ where $\phi(f)q=fq$ - see \cite{Lance} and \cite[I.1]{Williams} for details of the construction of the inner tensor product.)  Note that $P\otimes_{alg,C_{0}(Y)} Q$ is a dense Hilbert submodule of  $\ov{P}\otimes_{C_{0}(Y)} \ov{Q}$, so that $\ov{P}\otimes_{C_{0}(Y)} \ov{Q}=P\otimes_{C_{0}(Y)} Q$.
Canonically, $(P\otimes_{C_{0}(Y)}  Q)_{y}$ is the Hilbert space tensor product $P_{y}\otimes Q_{y}$ and for $p\in P, q\in Q$, 
$\widehat{p\otimes q}(y)=\hat{p}(y)\otimes \hat{q}(y)$.  We write $E_{P\otimes_{C_{0}(Y)}  Q}=E_{P}\otimes E_{Q}$.  (We note that this construction of the tensor product of two Hilbert bundles over $Y$ cannot be defined, as for vector bundles, using charts in the usual way (as, for example, in \cite[1.2]{AtiyahK}).)  

\begin{proposition}    \label{prop:tensGHilb}
If $P, Q$ are $G$-pre-Hilbert $C_{0}(Y)$-modules, then $P\otimes_{C_{0}(Y)} Q$ is a $G$-Hilbert module, the $G$-action being the diagonal one.
\end{proposition}
\begin{proof}
By definition of $P\otimes_{C_{0}(Y)} Q$, we can assume that $P, Q$ are $G$-Hilbert modules.  It is obvious that $G$ acts isometrically on 
$E_{P\otimes_{C_{0}(Y))}  Q}=\\
\cup_{y\in Y} P_{y}\otimes Q_{y}$.  For $G$-continuity, we only need to check Proposition~\ref{prop:cont} when 
$F = \hat{v}$ where $v=\sum_{i=1}^{n} p_{i}\otimes q_{i} \in P\otimes_{alg,C_{0}(Y)} Q$.  
Suppose then that $y_{r}\to y$ in $Y$, $g_{r}\to g$ in $G$ with $s(g_{r})=y_{r}$.   Since $F$ is continuous, 
$F(y_{r})\to \sum_{i=1}^{n} \hat{p_{i}}(y)\otimes \hat{q_{i}}(y)$.  Since $P, Q$ are $G$-Hilbert modules, for each $i$, $g_{r}\hat{p_{i}}(y_{r})\to g\hat{p_{i}}(y),    
g_{r}\hat{q_{i}}(y_{r})\to g\hat{q_{i}}(y)$ in $E_{P}, E_{Q}$ respectively.   Let $p'_{i}\in P, q'_{i}\in Q$ be such that $\widehat{p'_{i}}(r(g))=
g\widehat{p_{i}}(y)$, 
$\widehat{q'_{i}}(r(g))=g\widehat{q_{i}}(y)$, and set $w=\sum_{i=1}^{n}p'_{i}\otimes q'_{i}\in P\otimes_{alg,C_{0}(Y)} Q$.  Let $z_{r}=r(g_{r})$.  Then 
$\norm{\hat{w}(z_{r}) - \sum_{i=1}^{n} g_{r}\widehat{p_{i}}(y_{r})\otimes g_{r}\widehat{q_{i}}(y_{r})}
\leq \sum_{i=1}^{n}[\norm{\widehat{p'_{i}}(z_{r}) - g_{r}\widehat{p_{i}}(y_{r})}\norm{g_{r}\widehat{q_{i}}(y_{r})}
+ \norm{\widehat{p'_{i}}(z_{r})}\norm{\widehat{q'_{i}}(y_{r}) - g_{r}\widehat{q_{i}}(y_{r})}]\to 0$.  
Since $\hat{w}(z_{r})\to \hat{w}(r(g))=gF(y)$, $g_{r}F(y_{r})=\sum_{i=1}^{n} g_{r}\widehat{p_{i}}(y_{r})\otimes g_{r}\widehat{q_{i}}(y_{r})\to gF(y)$.  So 
$P\otimes_{C_{0}(Y)} Q$ is a $G$-Hilbert module.
\end{proof}

Next, we require the result that for any $G$-Hilbert modules $P, Q$, we have that as $G$-Hilbert modules, 
\begin{equation}    \label{eq:pqinfty}
(P^{\infty}\otimes_{C_{0}(Y)} Q)\cong (P\otimes_{C_{0}(Y)} Q^{\infty})\cong (P\otimes_{C_{0}(Y)} Q)^{\infty}.   
\end{equation}
Let us prove that $(P^{\infty}\otimes_{C_{0}(Y)} Q)\cong (P\otimes_{C_{0}(Y)} Q)^{\infty}$, the other equality being proved similarly.  Let $R$ be the  dense subspace of $P^{\infty}$ whose elements are the finite sequences $r=(p_{1},\ldots ,p_{n},0,0,\ldots)$ with $p_{i}\in P$.  Define a 
$C_{0}(Y)$-module map
$\al:R\otimes_{alg,C_{0}(Y)} Q\to (P\otimes_{C_{0}(Y)} Q)^{\infty}$ by setting $\al(r\otimes q)=
(p_{1}\otimes q,\ldots ,p_{n}\otimes q,0,0,\ldots)$.  It is easily checked that $\al$ is well-defined, and preserves the $C_{0}(Y)$-inner product: $\lan\al(r\otimes q),\al(r'\otimes q')\ran=\lan r,r'\ran\lan q,q'\ran =\lan r\otimes q,r'\otimes q'\ran$.  The range of $\al$ is onto a dense subspace of $(P\otimes_{C_{0}(Y)} Q)^{\infty}$ and preserves the $G$-action, so the result follows.

Of particular importance is the case of the $G$-Hilbert module $P\otimes_{C_{0}(Y)} P_{G}$.  We write 
$E_{P\otimes_{C_{0}(Y)} P_{G}}=L^{2}(G)\otimes E$ (or $E\otimes L^{2}(G)$) where $E=E_{P}$.  Here $L^{2}(G)\otimes E$ is the Hilbert bundle over $Y$ with 
$(L^{2}(G)\otimes E)_{y}=L^{2}(G^{y},E_{y})$ and a dense subspace of $C_{0}(Y, L^{2}(G)\otimes E)$, determining its topology as earlier, is given by the span of sections of the form $\hat{h}\otimes \hat{p}$ ($h\in C_{c}(G)$) where\\ $(\hat{h}\otimes \hat{p})(y)=h_{\mid G^{y}}\otimes \hat{p}(y)$.  A section $k$ of $L^{2}(G)\otimes E$
is invariant if for all $g\in G$, $gk_{s(g)}(g^{-1}h)=k_{r(g)}(h)$ ($h\in G^{r(g)}$) as maps in $L^{2}(G^{r(g)}, E_{r(g)})$.  We now identify a certain dense linear subspace $C_{c}(G,r^{*}E)$ of $C_{0}(Y, L^{2}(G)\otimes E)$ (cf. \cite{renrep}).    Here,  
$C_{c}(G,r^{*}E)$ is the set of continuous, compactly supported functions $\phi$ from $G$ into $E$ such that for all $g\in G$, $\phi(g)\in E_{r(g)}$.
For each $y\in Y$ and $\phi\in C_{c}(G,r^{*}E)$, let $\hat{\phi}(y)=\phi_{y}$, the restriction of $\phi$ to $G^{y}$.  Then 
$\hat{\phi}(y)\in C_{c}(G^{y},E_{y})\subset L^{2}(G^{y},E_{y})=(L^{2}(G)\otimes E)_{y}$ so that $\hat{\phi}$ is a section of $L^{2}(G)\otimes E$.  
The section norm on $C_{c}(G,r^{*}E)$ is then given by: $\norm{\phi}=\sup_{y\in Y}\norm{\phi_{y}}$.

\begin{proposition}   \label{prop:ppg}
Let $P$ be a $G$-Hilbert module and $E=E_{P}$.  Then $C_{c}(G,r^{*}E)^{\widehat{}}$ is a dense subspace of $C_{0}(Y, L^{2}(G)\otimes E)$, and contains all functions of the form $\hat{h}\otimes \hat{p}$ above.
\end{proposition}
\begin{proof}
Clearly, $\hat{h}\otimes \hat{p}\in C_{c}(G,r^{*}E)$ since the map $g\to h(g)\hat{p}(r(g))$ is continuous.  For the rest of the proposition, the span of such functions $\hat{h}\otimes \hat{p}$ is uniformly dense in $C_{0}(Y,L^{2}(G)\otimes E)$, so it is enough to show that every $\hat{\phi}$ 
($\phi\in C_{c}(G,r^{*}E)$) is in the uniform closure of this span.

To this end, let $H=supp(\phi)$.  Let $y_{0}\in Y$.  Let $W$ be a compact subset of $G$ such that $H\subset W^{0}$.  Let $\eps>0$.  For each $g\in H$, let $p_{g}\in P$ be such that $\widehat{p_{g}}(r(g))=\phi(g)$.  Let $h_{g}\in C_{c}(G)$ be such that $h_{g}(g)=1$.  By continuity, there exists an open neighborhood $U_{g}$ of $g$ in $G$ such that $U_{g}\subset W$ and such that for all $g'\in U_{g}$, 
\[  \norm{\phi(g') - h_{g}(g')\widehat{p_{g}}(r(g'))}<\eta=\eps/[\sup_{y\in Y}\la^{y}(W)^{1/2}+1].  \]
Since $H$ is compact, it is covered by a finite number of the $U_{g}$'s, say $U_{g_{1}}, \ldots ,U_{g_{n}}$.  Taking a partition of unity, there exist functions $f_{i}\in C_{c}(U_{g_{i}})$, $f_{i}\geq 0$,
$\sum_{i=1}^{n} f_{i}=1$ on $H$ and $\sum_{i=1}^{n} f_{i}\leq 1$ on $G$.  Then for $g'\in W$,\\ 
$\norm{\phi(g') - \sum_{i=1}^{n} f_{i}(g')h_{g_{i}}(g')\hat{p_{g_{i}}}(r(g'))}<\eta$.  It follows that for $y\in Y$,
\[     \norm{\phi_{y} - \sum_{i=1}^{n} (f_{i}h_{g_{i}}\otimes p_{g_{i}})_{y}}_{2}<\eps.     \]
So $\phi\in C_{0}(Y,L^{2}(G)\otimes E)$.
\end{proof}

We now note two simple results on the tensor products of two $G$-Hilbert modules.  First, if $P$ is a $G$-Hilbert module then 
\begin{equation}   \label{eq:YPP}
C_{0}(Y)\otimes_{C_{0}(Y)} P\cong P.
\end{equation}
The natural isomorphism is given by the equivariant Hilbert module map determined by: $f\otimes P\to fp$ ($f\in C_{0}(Y), p\in P$).  Next, it is left to the reader to check that if $P, Q, R$ are $G$-Hilbert modules, then the Hilbert module direct sum $P\oplus Q$ is a $G$-Hilbert module in the obvious way, and 
\begin{equation}  \label{eq:PQR}
(P\oplus Q)\otimes_{C_{0}(Y)} R\cong (P\otimes_{C_{0}(Y)} R)\oplus (Q\otimes_{C_{0}(Y)} R).
\end{equation}
The final proposition of this section is a groupoid version of \cite[Lemma 2.3]{MingoPhillips} (which applies to the group case).  

\begin{proposition}  \label{prop:Gcong}
Let $P, Q$ be $G$-Hilbert modules with $P\cong Q$ as Hilbert $C_{0}(Y)$-modules.  Then 
$P\otimes_{C_{0}(Y)} P_{G}\cong Q\otimes_{C_{0}(Y)} P_{G}$ as $G$-Hilbert modules.
\end{proposition}
\begin{proof}
Let $E=E_{P}, F=E_{Q}$.  By assumption, there exists a Hilbert module unitary $U:P\to Q$.  For $\phi\in C_{c}(G,r^{*}E)$, define $V\phi:G\to r^{*}F$ by:
\begin{equation}   \label{eq:Vphi}
V\phi(g)=gU_{s(g)}(g^{-1}\phi(g)).
\end{equation}
Using the continuity of $\Phi_{U}=\{U_{y}\}$ and of the $G$-actions of $E, F$, we see that $V\phi$ belongs to $C_{c}(G,r^{*}F)$.
Regard, as earlier, $C_{c}(G,r^{*}E), C_{c}(G,r^{*}F)$ fibered over $Y$ (with $\phi\to \{\phi_{y}\}$).  Then $V$ is a fiber preserving isomorphism onto  $C_{c}(G,r^{*}F)$  with $V^{-1}\chi(g)=gU_{s(g)}^{*}(g^{-1}\chi(g))$.  Further,
\[ \lan (V\phi)_{y},(V\psi)_{y}\ran = \int \lan gU_{s(g)}(g^{-1}\phi_{y}(g)), gU_{s(g)}(g^{-1}\psi_{y}(g))\ran\,d\la^{y}(g)=\lan \phi_{y},\psi_{y}\ran  \]
so that $V$ preserves inner products.  So $V$ extends to a Hilbert module unitary from $C_{0}(Y,L^{2}(G)\otimes E)\to C_{0}(Y,L^{2}(G)\otimes F)$, using Proposition~\ref{prop:ppg} and Proposition~\ref{prop:Etop}.  It remains to show that $V$ is $G$-equivariant.  We note first that by (\ref{eq:Vphi}),
$V_{y}$ is given by: $V_{y}\xi(g)=gU_{s(g)}(g^{-1}\xi(g))$ for $\xi\in L^{2}(G^{y},E_{y}), g\in G^{y}$.  Then for $g, h\in G^{y}$,
$[g(V_{s(g)}\xi_{s(g)})](h)=g[V_{s(g)}\xi_{s(g)}(g^{-1}h)]=\\
g(g^{-1}h)[U_{s(h)}((g^{-1}h)^{-1}\xi_{s(g)}(g^{-1}h))]=
h[U_{s(h)}(h^{-1}[(g\xi_{s(g)})(h)])]=\\
V_{r(g)}(g\xi_{s(g)})(h)$, so that $gV_{s(g)}=V_{r(g)}g$ and $V$ is equivariant.
\end{proof}

\section{Stabilization}
In this section we establish the proper groupoid stabilization theorem.  Throughout, $G$ is a proper groupoid and $P$ a $G$-Hilbert module.  We require two preliminary propositions.  The first of these is the general groupoid version of \cite[Lemma 2.8]{Phillips}.

\begin{proposition} \label{prop:myphi}
There exists a continuous, invariant section $\phi$ of the Hilbert bundle $L^{2}(G)^{\infty}$ such that $\norm{\phi(y)}_{2}=1$ for all $y$.  Locally, 
$\phi(y)$ is of the form 
\[  ((\psi_{1})_{\mid G^{y}},\ldots ,(\psi_{n})_{\mid G^{y}},0,\ldots)   \]
where $\psi_{i}\in C_{c}(G)$.
\end{proposition}
\begin{proof}
For $y_{0}\in Y$, let $a_{y_{0}}\in C_{c}(G)$ be such that $a_{y_{0}}\geq 0, a_{y_{0}}(y_{0})>0$.  Let $\eta_{y_{0}}:G\to \R^{+}$ be given by:
\[  \eta_{y_{0}}(g)=\int _{G^{r(g)}}a_{y_{0}}(h^{-1}g)\,d\la^{r(g)}(h).                      \]
We want to regard $k=\eta_{y_{0}}$ as a continuous, invariant section $y\to k_{y}$ of $L^{2}(G)$.  To prove this, the invariance of $k$ 
(i.e. that $g_{0}k_{s(g_{0})}=k_{r(g_{0})}$, or equivalently, that $k(g_{0}^{-1}g)=k(g)$ for all $g_{0},g\in G, r(g_{0})=r(g)$) follows from an axiom for left Haar systems.  For the continuity of the section $y\to k_{y}$ of $L^{2}(G)$, we will show that for any compact subset $A$ of $Y$, 
$k_{\mid r^{-1}A}\in C_{c}(r^{-1}A)$.  The continuity of $k$ as a section of $L^{2}(G)$ then follows, since for every relatively compact open subset $U$ of $Y$, there will then exist an $F\in C_{c}(G)$ such that $F=k$ on $r^{-1}U$ (so that $F_{y}=k_{y}$ for all $y\in U$).  
Since $F$ is continuous as a section $y\to F_{y}$ of $L^{2}(G)$, so also is $k$.  (Of course, $y\to k_{y}$ need not vanish at infinity.)

To show that $k_{\mid r^{-1}A}\in C_{c}(r^{-1}A)$, let $C$ be the (compact) support of $a_{y_{0}}$ and let $g\in r^{-1}A$.  If $a_{y_{0}}(h^{-1}g)>0$, then 
$r(h)=r(g)\in A$, and $s(h)\in r(C)$.  By the properness of $G$, $h$ belongs to the compact set $D=\{g'\in G: (r(g'),s(g'))\in A\x r(C)\}$.   Let $F\in C_{c}(G)$ be such that $F=1$ on $D$.  Then on $r^{-1}(A)$, $k$ coincides with the convolution $F*a_{y_{0}}$ of two $C_{c}(G)$-functions, and so is the restriction of a $C_{c}(G)$-function as required.  

By the continuity and positivity assumptions on $a_{y_{0}}$, the function $\eta_{y_{0}}(y_{0})> 0$.  So $(\eta_{y_{0}})_{y_{0}}\ne 0$.
By the continuity of $y\to \norm{(\eta_{y_{0}})_{y}}_{2}$, the set $U_{y_{0}}=\{y\in Y: (\eta_{y_{0}})_{y}\ne 0\}$ is an open neighborhood of $y_{0}$ in $Y$.  Since
$\eta_{y_{0}}$ is invariant, it follows that $U_{y_{0}}$ is an invariant subset of $Y$, i.e. is such that for $g\in G$, $s(g)\in U_{y_{0}}$ if and only if 
$r(g)\in U_{y_{0}}$.   Further, the $U_{y_{0}}$'s cover $Y$.  Since the action of $G$ on $Y$ is proper, there is a $G$-partition of unity $\{f_{\ga}: \ga\in S\}$, where $S$ can be taken to be infinitely countable (and so identified with $\{1,2,3,\ldots\}$), subordinate to the $U_{y}$'s (\cite{Patanindex}, \cite[Proposition 4]{Patequivar}).  This means that for each $\ga$,  $f_{\ga}\in C_{c}(Y)$, $0\leq f_{\ga}$, there exists a $y(\ga)\in Y$ such that $supp(f_{\ga})\subset U_{y(\ga)}$, and with $m_{\ga}:Y\to \R$ given by
$m_{\ga}(y)=\int_{G^{y}}f_{\ga}(s(g))\,d\la^{y}(g)$, we have
\begin{equation}   \label{eq:pou}
\sum_{\ga} m_{\ga}(y)=1,
\end{equation} 
the sum being locally finite.  

Using the properness of $G$ and the continuity of the maps $y\to \int_{G^{y}} F(g)\,d\la^{y}(g)$ for $F\in C_{c}(G)$, $m_{\ga}$ is invariant 
(i.e. $m_{\ga}(s(g))=m_{\ga}(r(g))$ for all $g\in G$) and continuous.  Define a section $\phi=\{\phi_{\ga}\}$ of $L^{2}(G)^{\infty}$ by setting 
\[  \phi_{\ga}(y)=m_{\ga}(y)^{1/2}(\norm{(\eta_{y(\ga)})_{\mid G^{y}}}_{2})^{-1}(\eta_{y(\ga)})_{\mid G^{y}}.   \]
We take $\phi_{\ga}(y)$ to be $0$ whenever $(\eta_{y(\ga)})_{\mid G^{y}}=0$.  For continuity reasons, we need to know that if 
$(\eta_{y(\ga)})_{\mid G^{y}}=0$ then $m_{\ga}(y)=0$.  To prove this, suppose then that $(\eta_{y(\ga)})_{\mid G^{y}}=0$.  Then 
$y\in Y\setminus U_{y(\ga)}$, which is invariant since $U_{y(\ga)}$ is.  So if $g\in G^{y}$, then $s(g)\in Y\setminus U_{y(\ga)}$, and in that case, 
$f_{\ga}(s(g))=0$ (since the support of $f_{\ga}$ lies inside $U_{y(\ga)}$), so that $m_{\ga}(y)=0$ from the definition of $m_{\ga}$.  

We now claim that $\phi=\{\phi_{\ga}\}$ is continuous and $G$-invariant. For the continuity of $\phi$, we note
that $\norm{\phi_{\ga}(y)}_{2}^{2}=m_{\ga}(y)$, and use the preceding paragraph, the local finiteness of the sum in (\ref{eq:pou}), and the 
continuity of the maps $y\to (\eta_{y(\ga)})_{\mid G^{y}}$, $y\to \norm{(\eta_{y(\ga)})_{\mid G^{y}}}_{2}$ to obtain that 
locally $\phi$ takes values in some $L^{2}(G)^{n}$ with $n$ finite and components the restrictions of $C_{c}(G)$-functions.  Since the $\eta_{y(\ga)}$, $m_{\ga}$ are $G$-invariant so also is $\phi$.  

Last, from (\ref{eq:pou}), $\norm{\phi(y)}_{2}=[\sum_{\ga\in S}\norm{\phi_{\ga}(y)}_{2}^{2}]^{1/2}=1$. 
\end{proof}

\begin{proposition}   \label{prop:ppl}
\begin{equation}  \label{eq:PPP}
P\oplus (P\otimes_{C_{0}(Y)} P_{G}^{\infty})\cong P\otimes_{C_{0}(Y)} P_{G}^{\infty}.   
\end{equation}
\end{proposition}
\begin{proof} 
Let $\phi$ be the continuous, invariant section of $L^{2}(G)^{\infty}$ given by Proposition~\ref{prop:myphi}.  For each $p\in P$, define a section $W\hat{p}$ of
$E\otimes L^{2}(G)^{\infty}$ by: $W\hat{p}=\hat{p}\otimes \phi$.  We claim that $W\hat{p}\in C_{0}(Y,E\otimes L^{2}(G)^{\infty})$.  To prove that $\norm{W\hat{p}(y)}\to 0$ as $y\to \infty$, by Proposition~\ref{prop:myphi}, $\norm{\hat{p}\otimes \phi(y)}_{2}=\norm{\hat{p}(y)}_{2}\norm{\phi(y)}_{2}=\norm{\hat{p}(y)}_{2}\to 0$ as $y\to \infty$.  The continuity of $W\hat{p}$ follows from the fact that locally, it is the restriction of an element of $(P\otimes_{alg,C_{0}(Y)} C_{c}(G)^{n})^{\widehat{}}$ (which is a subspace of the space of continuous sections of $E \otimes L^{2}(G)^{\infty}$).  It is easy to check that $W:C_{0}(Y,E)\to C_{0}(Y,E\otimes L^{2}(G)^{\infty})$ is a linear, $C_{0}(Y)$-module map, and that $\lan W\hat{p}, W\hat{p'}\ran = \lan\hat{p}, \hat{p'}\ran$.  Further, using the invariance of $\phi$,
\[ W_{r(g)}(g\hat{p}_{s(g)})=g\hat{p}_{s(g)}\otimes \phi_{r(g)}=g\hat{p}_{s(g)}\otimes g\phi_{s(g)}=gW(\hat{p})_{s(g)}
=gW_{s(g)}(\hat{p}_{s(g)})              \]
so that $W$ is $G$-invariant.  By Proposition~\ref{prop:Etop}, there exists a map 
$V:P\to P\otimes_{C_{0}(Y)} P_{G}^{\infty}$ such that $\widehat{Vp}=W\hat{p}$.  Note that $V_{y}p_{y}=(W\hat{p})_{y}$.

From the corresponding properties for $W$, $\lan V(p), V(p')\ran=\lan p,p'\ran$ and $V$ is a $G$-equivariant $C_{0}(Y)$-module map.
We claim that $V$ is adjointable with adjoint $V^{*}$ determined by: $V^{*}(p\otimes \psi)=p\lan \phi, \hat{\psi}\ran$ for
$\psi\in \cup_{n=1}^{\infty}C_{c}(G)^{n}$, a dense subspace of $P_{G}^{\infty}$.  Note that by definition, 
$\lan \phi, \hat{\psi}\ran(y)=\lan \phi_{y},\psi_{y}\ran$ (the inner product evaluated in $L^{2}(G^{y})^{\infty}$),
and using Proposition~\ref{prop:myphi}, 
$\lan \phi, \hat{\psi}\ran\in C_{0}(Y)$ and $\mid \lan \phi, \hat{\psi}\ran(y)\mid \leq \norm{\psi_{y}}$.  Now
\[  \lan p\otimes\psi, Vp'\ran=\lan \hat{p}\otimes \hat{\psi},\widehat{Vp'}\ran = \lan \hat{p},\hat{p'}\ran\lan \hat{\psi},\phi\ran
=\lan p\lan\phi,\hat{\psi}\ran ,p'\ran .  \]
It is easy to check that the bilinear map 
$p\otimes \psi\to p\lan \phi, \hat{\psi}\ran$ extends to a linear map $V^{*}$ from $P\otimes_{alg,C_{0}(Y)} P_{G}^{\infty}\to P$, and so
$\lan t,Vp'\ran = \lan V^{*}t,p'\ran$ for all $t\in P\otimes_{alg,C_{0}(Y)} P_{G}^{\infty}, p'\in P$.  Since
$\norm{\lan V^{*}(\sum_{i=1}^{n}p_{i}\otimes \psi_{i}),p'\ran}=\norm{\lan \sum_{i=1}^{n}p_{i}\otimes \psi_{i},Vp'\ran}
\leq \norm{\sum_{i=1}^{n}p_{i}\otimes \psi_{i}}\norm{p'}$, $V^{*}$ is continuous on $P\otimes_{alg,C_{0}(Y)} P_{G}^{\infty}$ and so
extends by continuity to $P\otimes_{C_{0}(Y)} P_{G}^{\infty}$.  This extension is the adjoint of $V$ as claimed.

Using the approach of Mingo and Phillips (\cite{MingoPhillips}), define \\ 
$U:P\oplus (P\otimes_{C_{0}(Y)} P_{G}^{\infty})\to (P\otimes_{C_{0}(Y)} P_{G}^{\infty})$ by:
\[  U(p_{0},\xi_{1},\xi_{2},\ldots )=(Vp_{0}+(1-VV^{*})\xi_{1},VV^{*}\xi_{1}+(1-VV^{*})\xi_{2},\ldots  ). \]
One checks that for each $w=(p_{0},\xi)$, $U(w)=(b_{1}, b_{2}, \ldots )$ belongs to \\
$(P\otimes_{C_{0}(Y)} P_{G}^{\infty})^{\infty}$, i.e.
that $\sum_{i=1}^{\infty} \lan b_{i},b_{i}\ran$ converges in $C_{0}(Y)$.  
By (\ref{eq:pinftyinfty}) and (\ref{eq:pqinfty}), $(P\otimes_{C_{0}(Y)} P_{G}^{\infty})^{\infty}= P\otimes_{C_{0}(Y)} P_{G}^{\infty}$. 
Further, $U$ preserves the $C_{0}(Y)$-valued inner product.  Direct calculation shows that $U$ has an adjoint given by:
\[  U^{*}(\eta_{1},\eta_{2},\ldots )=(V^{*}\eta_{1},VV^{*}\eta_{2}+(1-VV^{*})\eta_{1},VV^{*}\eta_{3}+(1-VV^{*})\eta_{2}, \ldots ),   \]
that $U$ is unitary and, using the invariance of $V$, that $U$ preserves the groupoid action. 
\end{proof}

\begin{theorem}   \label{th:stab}
(Groupoid stabilization theorem)\hspace{.2in}  If $P$ is a $G$-Hilbert module, then 
\[   P\oplus P_{G}^{\infty}\cong P_{G}^{\infty}.           \]
\end{theorem}
\begin{proof} 
We claim first that 
\begin{equation}   \label{eq:GPGG}
P_{G}^{\infty}\cong (P\otimes_{C_{0}(Y)} P_{G}^{\infty})\oplus P_{G}^{\infty}.
\end{equation}  

For using (\ref{eq:YPP}), (\ref{eq:pinftyinfty}), (\ref{eq:pqinfty}),  
the non-equivariant stabilization theorem, Proposition~\ref{prop:Gcong} and (\ref{eq:PQR}),
\[ P_{G}^{\infty}\cong (C_{0}(Y)\otimes_{C_{0}(Y)} P_{G})^{\infty}\cong C_{0}(Y)^{\infty}\otimes_{C_{0}(Y)} P_{G}^{\infty}
\cong (P\oplus C_{0}(Y)^{\infty})\otimes_{C_{0}(Y)} P_{G}^{\infty} \]
\[ \cong (P\otimes_{C_{0}(Y)} P_{G}^{\infty})\oplus (C_{0}(Y)^{\infty}\otimes_{C_{0}(Y)} P_{G}^{\infty})\cong 
(P\otimes_{C_{0}(Y)} P_{G}^{\infty})\oplus P_{G}^{\infty}.\]
Using (\ref{eq:GPGG}) and (\ref{eq:PPP}),
\[ P\oplus P_{G}^{\infty}\cong P\oplus ((P\otimes_{C_{0}(Y)} P_{G}^{\infty})\oplus P_{G}^{\infty})=  
[P\oplus (P\otimes_{C_{0}(Y)} P_{G}^{\infty})]\oplus P_{G}^{\infty}  \]
\[  \cong (P\otimes_{C_{0}(Y)} P_{G}^{\infty})\oplus P_{G}^{\infty}
\cong P_{G}^{\infty}.    \]
\end{proof}

\end{document}